\documentclass[11pt]{amsart}

\usepackage{amsthm,amssymb,amsfonts,amsmath,enumerate,graphicx,color}
\usepackage{mathrsfs}
\usepackage[margin=1in]{geometry}

\newcommand{\LT}{\textsc{lt}}
\newcommand{\api}{\hat{a}_{\pi}}
\newcommand{\asigma}{\hat{a}_{\sigma}}
\newcommand{\Hilb}[2]{{\rm Hilb}(#1; \ #2)}

\newcommand{\NNeg}{\mathrm{NNeg}}
\newcommand{\NDes}{\mathrm{NDes}}
\newcommand{\ndes}{\mathrm{ndes}}
\newcommand{\nmajor}{\mathrm{nmajor}}

\newcommand{\fdes}{\mathrm{fdes}}
\newcommand{\fmajor}{\mathrm{fmajor}}

\newcommand{\Pc}{\mathcal{P}}
 
\newcommand{\C}{\mathbb{C}}
\newcommand{\Z}{\mathbb{Z}}
\newcommand{\N}{\mathbb{N}}
\newcommand{\R}{\mathbb{R}}

\newcommand{\Des}{\mathrm{Des}}
\newcommand{\maj}{\mathrm{maj}}
\newcommand{\des}{\mathrm{des}}

\newcommand{\cone}{\mathrm{cone}}

\renewcommand{\phi}{\varphi}
\renewcommand{\emptyset}{\varnothing}

\def\x{{\boldsymbol x}}

\newcommand\commentout[1]{}

\newtheorem{theorem}{Theorem}[section]

\newtheorem{proposition}[theorem]{Proposition}

\theoremstyle{remark}
\newtheorem{example}[theorem]{Example}
\newtheorem{remark}[theorem]{Remark}
\theoremstyle{definition}
\newtheorem{definition}[theorem]{Definition}

\newtheorem{algorithm}[theorem]{Algorithm}

\begin{document}

\title[Euler-Mahonian statistics and descent bases]{Euler-Mahonian statistics and descent bases for semigroup algebras}

\author{Benjamin Braun}
\address{Department of Mathematics\\
         University of Kentucky\\
         Lexington, KY 40506--0027}
\email{benjamin.braun@uky.edu}

\author{McCabe Olsen}
\address{Department of Mathematics\\
         University of Kentucky\\
         Lexington, KY 40506--0027}
\email{mccabe.olsen@uky.edu}


% 05A05 Permutations Words Matrices
% 05A19 Combinatorial Identities, Bijective Combinatorics
% 05E05 Symmetric Functions and Generalizations
% 05E40 Combinatorial Aspects of Commutative Algebra
% 13P10 Groebner bases, other bases for ideals and modules
% 52B20 Lattice Polytopes (including relations with commutative algebra and algebraic geometry)

\date{9 October 2017}

\thanks{
The authors would like to thank the anonymous referees for their helpful comments and suggestions. The first author was partially supported by grant H98230-16-1-0045 from the U.S. National Security Agency.
The second author was partially supported by a 2016 National Science Foundation/Japanese Society for the Promotion of Science East Asia and Pacific Summer Institutes Fellowship award NSF OISE--1613525.
}

%%%%%%%%%%%%%%%%%%%%%%%%%%%%%%%%%%%%%%%%%%%%%%%%%%%%%%%%%%%%%%

\begin{abstract}
We consider quotients of the unit cube semigroup algebra by particular $\Z_r\wr S_n$-invariant ideals. 
Using Gr\"obner basis methods, we show that the resulting graded quotient algebra has a basis where each element is indexed by colored permutations $(\pi,\epsilon)\in\Z_r\wr S_n$ and each element encodes the negative descent and negative major index statistics on $(\pi,\epsilon)$. 
This gives an algebraic interpretation of these statistics that was previously unknown. 
This basis of the $\Z_r\wr S_n$-quotients allows us to recover certain combinatorial identities involving Euler-Mahonian distributions of statistics. 
\end{abstract}

\maketitle

%%%%%%%%%%%%%%%%%%%%%%%%%%%%%%%%%%%%%%%%%%%%%%%%%%%%%%%%%%%

\section{Introduction}\label{Intro}
Let $[0,1]^n\subset \R^n$ denote the  the $n$-dimensional unit cube.
Let $S_n$ denote the symmetric group on $n$ elements. 
Let $[n]:=\{1,2,\ldots,n\}$.

\subsection{Polytope semigroup algebras}
Let $\Pc\subset \R^n$ be an $n$-dimensional convex lattice polytope, let $m\cdot \Pc=\{m\alpha : \alpha\in \Pc\}$ denote the $m$th dilate of $\Pc$,  and consider the \emph{cone over} $\Pc$
	\[
	\cone(\Pc):={\rm span}_{\R\geq 0}\{(1,p) \, : \, p\in\Pc\} \, .
	\]
The {\it affine semigroup algebra of} $\Pc$ over $\C$ is 
	\[
	\C[\Pc]:=\C[t^m\cdot\x^{p} :(m,p)\in \cone(\Pc)\cap\Z^{n+1} ]\subset \C[t, x_1^{\pm 1},x_2^{\pm 1},\ldots, x_n^{\pm 1}] \, ,
	\]
where $\x^{p}=x_1^{p_{1}}x_2^{p_{2}}\cdots x_n^{p_{n}}$ when $(m,p)\in \cone(\Pc)\cap\Z^{n+1}$. 
Given that $\cone(\Pc)$ is a pointed, rational cone in $\R^{n+1}$,  $\cone(\Pc)\cap \Z^{n+1}$ has a unique minimal generating set called a \emph{Hilbert basis}. 
Subsequently, the algebra $\C[\Pc]$ is a finitely generated, graded commutative algebra. 
If $\Pc$ satisfies the \emph{integer decomposition property}, that is for any $q\in m\cdot\Pc \cap \Z^n$, we can express $q=q_1+q_2+\cdots+q_m$ where each $q_i\in \Pc\cap\Z^n$, then we can more concisely describe $\C[\Pc]$. 
In particular, if $\Pc\cap\Z^n=\{ p_1,p_2,\ldots,p_k\}$ and $\Pc$ satisfies the integer decomposition property, then
	\[
	\C[\Pc]:=\C[t\cdot\x^{p_i} \, : \, 1\leq i\leq k]\subset \C[t, x_1^{\pm 1},x_2^{\pm 1},\ldots, x_n^{\pm 1}] \, .
	\]
For greater detail and background of semigroup algebras and cones over polytopes, see \cite{MillerSturmfels}.

Let $\Pc=[0,1]^n$, which is known to satisfy the integer decomposition property. 
Let $R_n:=\C\left[ [0,1]^n \right]$ denote the affine semigroup algebra of $[0,1]^n$  which has the following description:
	\[
	R_n= \C\left[t\cdot x_{a_1}\cdots x_{a_i} \ | \ A=\{a_1,\ldots,a_i\}\subseteq [n] \right] \subset \C[t, x_1,x_2,\ldots,x_n] \, .
	\]
Alternatively, we can define $R_n$ as the quotient of a polynomial ring by a toric ideal. 
Let $T_n$ be a polynomial ring in $2^n$ variables, where each variable corresponds to a subset of $[n]$, thus
\[
T_n:=\C\left[z_A: A\subseteq [n]\right] \, .
\]
Define the toric ideal  
	\[
	I_n:=\langle z_Az_B-z_{A\cap B}z_{A\cup B} \ | \  A\not\subseteq B \mbox{ and } B\not\subseteq A  \rangle.
	\]
It is known that $R_n\cong T_n/I_n$. For background and details see \cite{MillerSturmfels}. 
This algebra also arises as the Hibi ring for the antichain on $n$ elements, as the unit cube is the order polytope of the antichain (see e.g. \cite{Ene,Hibi,ACCP} for additional details of Hibi rings). 
We will use $R_n$ to denote $T_n/I_n$ when it is convenient. 

Let $\mathcal{A}=\bigoplus_{b \in\Z^n} \mathcal{A}_b$ be a finitely generated,  $\Z^n$-graded commutative $\C$-algebra. 
The \emph{Hilbert series} of $\mathcal{A}$ is 
	\begin{equation}\label{HilbertSeries}
	\Hilb{\mathcal{A}}{\mathbf{z}}=\sum_{b\in\Z^n} \dim_\C(\mathcal{A}_b)\cdot \mathbf{z}^b.
	\end{equation}
For a polytope semigroup algebra $\C[\Pc]$, it is common to consider $\C[\Pc]$ as an $\N$-graded algebra where the grading is given by the $t$-degree. In this case, we have 
	\[
	\Hilb{\C[\Pc]}{t}=\sum_{m\geq 0}\#\left(m \Pc\cap \Z^n\right)\cdot t^m 
	\]
which coincides with the \emph{Ehrhart series} of $\Pc$. 
The reader is invited to consult \cite{MillerSturmfels} and \cite{BeckRobins} for background on Hilbert series and Ehrhart theory receptively.
In the case of $R_n$, it is well-known that the Hilbert series with respect to the $t$-degree is $\sum_{k\geq 0}(k+1)^nt^k$, leading us to the topic of Euler-Mahonian identities.

%%%%%%%%%%%%%%%%%%%%%%%%%%%%%%%%%%%%%%%

\subsection{Euler-Mahonian identities}
We first review certain permutation statistics.
For $\pi \in S_n$ with $\pi=\pi_1\pi_2\cdots\pi_n$,  the {\it descent set} is defined to be 
	\[
	\Des (\pi):=\{i\in \{1,2,\dots, n-1 \}: \pi_i> \pi_{i+1} \}.
	\]
Moreover, the {\it descent number} is $	\des (\pi):=\# \Des (\pi)$.
The descent number is encoded in the \emph{Eulerian polynomial} $A_n(t):=\sum_{\pi\in S_n}t^{\des (\pi)}$ which satisfies the identity 
	\begin{align}
	\sum_{k\geq 0}(k+1)^nt^k=\frac{A_n(t)}{(1-t)^{n+1}} \, ,
	\end{align}		
first studied by Euler \cite{Euler}. 
This identity was generalized to a bivariate identity usually attributed to Carlitz using the major index; see \cite{B-B} and the references therein for more details on the history of these identities. 
Given $\pi\in S_n$, the {\it major index} of $\pi$ is defined to be 
	\[
	\maj (\pi):=\sum_{j\in\Des(\pi)}j.
	\] 

\begin{theorem}[Carlitz, \cite{Carlitz}]\label{EMI}
	For all $n\geq 1$, 
	\[
	\sum_{k\geq 0}[k+1]_q^n t^k=\frac{\sum_{\pi\in S_n} t^{\des(\pi)}q^{\maj(\pi)}}{\prod_{j=0}^n(1-tq^j)}
	\]
	where $[k+1]_q=1+q+q^2+\cdots+q^k$.
\end{theorem}

In this form, this identity is due to Carlitz \cite{Carlitz}, though with some effort one can derive it from the works of MacMahon \cite[Volume 2, Chapter IV, \S462]{macmahon}.
We will call this identity the {\it Euler-Mahonian identity}, which has arisen in a variety of contexts in recent years. 
Some such scenarios include lecture hall partition generating function identities \cite{P-S}, polyhedral-geometric studies of the semigroup algebra for $\cone([0,1]^n)$ \cite{B-B}, Hilbert series related to a descent basis for the coinvariant algebra of $S_n$ \cite{A-B-R}, 0-Hecke algebra actions on Stanley-Reisner rings \cite{Huang}, and quasisymmetric function identities \cite{S-Wachs}.

Generalizing to colored permutation groups $\Z_r\wr S_n$, one can consider the \emph{flag statistics} as well as the {\it negative statistics}, the latter of which we define in Section~\ref{sec:stats}.
These statistics were orginally introduced for the hyperoctohedral group $B_n\cong \Z_2\wr S_n$ \cite{A-B-R1} and generalized for $r\geq 2$ to $\Z_r \wr S_n$ \cite{Bagno,Bagno-Biagioli}.
For these families of statistics, the following Euler-Mahonian identities exist.

\begin{theorem}[Bagno, \cite{Bagno}]\label{neg}
Given any $r\geq 2$, for all $n\geq 1$, 
	\[
	\sum_{k\geq 0}[k+1]_q^n t^k=\frac{\sum_{(\rho,\epsilon)\in \Z_r \wr S_n}t^{\ndes(\rho,\epsilon)}q^{\nmajor(\rho,\epsilon)}}{(1-t)\prod_{j=1}^n(1-t^rq^{rj})}
	\] 
\end{theorem}

\begin{theorem}[Bagno-Biagioli, \cite{Bagno-Biagioli}]\label{flag}
Given any $r\geq 2$, for all $n\geq 1$,
	\[
	\sum_{k\geq 0}[k+1]_q^n t^k=\frac{\sum_{(\rho,\epsilon)\in \Z_r \wr S_n}t^{\fdes(\rho,\epsilon)}q^{\fmajor(\rho,\epsilon)}}{(1-t)\prod_{j=1}^n(1-t^rq^{rj})}
	\]
\end{theorem}

\subsection{Our Contributions}

The goal of this paper is twofold. First, we produce a new algebraic interpretation of negative permutation statistics by considering $\Z_r\wr S_n$-quotient algebras of $R_n$. 
To do so, we consider an ideal $\overline{invar(r,n)}\subset R_n$ which is generated by certain invariants of $R_n$ under a $\Z_r\wr S_n$-action, defined in detail in Section~\ref{quotients}. 
We obtain the following theorem using Gr\"obner basis techniques.   

\begin{theorem} [see Theorem \ref{basistheorem}]
There exists a basis of $R_n/\overline{invar(r,n)}$ of the form $\{b^r_{(\sigma,X)}+\overline{invar(r,n)}\}$ with elements indexed by pairs $(\sigma,X)$ that are in bijection with colored permutations $(\pi,\epsilon)\in\Z_r\wr  S_n$.
Further, $b^r_{(\sigma,X)}$ encodes $\ndes(\pi,\epsilon)$ and $\nmajor(\pi,\epsilon)$. 
The bijective correspondence of $(\sigma,X)\leftrightarrow(\pi,\epsilon)$ is given in Remark \ref{notation}.
\end{theorem}

Our second goal is to consider a multigraded Hilbert series of $R_n$ and the quotient $R_n/\overline{invar(r,n)}$. These computations allow us to recover the identities given by Theorem \ref{EMI} and Theorem \ref{neg}.
 These new proofs provide a new perspective on identities of this type. 

Moreover, the new proof of Theorem \ref{EMI} serves to expand connections between the commutative-algebraic and representation-theoretic methods \cite{A-B-R} for the $S_n$-coinvariant algebra $\C[x_1,\cdots, x_n]/\mathscr{I}_n$, where $\mathscr{I}_n:=<e_1,\ldots,e_n>$ with $e_i$ denoting the $i$-th elementary symmetric function, and polyhedral-geometric methods for $\cone([0,1]^n)$ \cite{B-B}. Additionally, we provide a short proof that this quotient algebra is isomorphic as a graded $S_n$-module to the  $S_n$-coinvariant algebra $\C[x_1,\cdots, x_n]/\mathscr{I}_n$.
We believe that these results, like those given in \cite{B-B}, support the idea that $\cone([0,1]^n)$ and its associated semigroup algebra are analogues of the polynomial ring in $n$ variables that give rise to interesting and different structures and results in similar contexts.

%%%%%%%%%%%%%%%%%%%%%%%%%%%%%%%%%%%%%%%%%%%%%%%%%%%%%%%%%%%

\section{Colored permutation groups and decent sets} \label{sec:stats}

The wreath product $\Z_r \wr S_n\cong (\Z_r)^n\ltimes S_n$ of a cyclic group of order $r$ with $S_n$ consists of pairs $(\pi,\epsilon)$ where $\pi\in S_n$ and $\epsilon\in \{\omega^0,\omega^1,\ldots,\omega^{r-1}\}^n$ for $\omega:=e^{2\pi i/r}$ a primitive $r$th root of unity. 
These groups are often called {\it colored permutation groups} and the elements are commonly refered to as {\it colored} or {\it indexed} permutations. 
We adopt the usual window notation, denoting the pair $(\pi, \epsilon)$ by $[\pi(1)^{c_1} \ \pi(2)^{c_2} \ \cdots \ \pi(n)^{c_n}]$ where $\epsilon_j=\omega^{c_j}$. 
Additionally, we will use the notation $j^{c_j}$ and $(\omega^{c_j},j)$ to denote elements of $\{\omega^0,\omega^1,\ldots,\omega^{r-1}\}\times [n]$. 

Elements $(\pi,\epsilon)\in\Z_r\wr S_n$ can be identified as a permutation matrix for $\pi$ where the 1 in position $(\pi(i),i)$ is replaced with $\epsilon_i$. 
The algebraic structure of $\Z_r\wr S_n$ is described by matrix multiplication where entry-by-entry multiplication of the nonzero entries is given by the group operation of $\Z_r$. This means that given $(\pi,\epsilon), (\pi',\epsilon')\in\Z_r\wr S_n$
	\[
	(\pi',\epsilon')\circ (\pi,\epsilon)=(\pi'\circ\pi, (\epsilon_1 \cdot\epsilon'_{\pi(1)}, \ldots, \epsilon_n\cdot \epsilon'_{\pi(n)})),
	\]
or represented in window notation we have 
	\[
	[\pi'(1)^{c'_1} \  \cdots \ \pi'(n)^{c'_n}]\circ [\pi(1)^{c_1} \  \cdots \ \pi(n)^{c_n}]=[\pi'\circ\pi(1)^{c_1+c'_{\pi(1)}} \  \cdots \ \pi'\circ\pi(n)^{c_n+c'_{\pi(n)}}]
	\]
where the addition is modulo $r$. A more explicit understanding of these wreath products may be found in \cite{Bagno,Bagno-Biagioli,B-B,P-S}.

To review one definition of descents for wreath products, we define a total order as follows.
Given $j^{c_j},k^{c_k}\in\{\omega^0,\omega^1,\cdots,\omega^{r-1}\}\times [n]$ we say that $j^{c_j}<k^{c_k}$ if $c_j>c_k$ or if $c_j=c_k$ and $j<k$ hold. 

\begin{definition}
Let $(\pi,\epsilon)\in\Z_r\wr S_n$.  The {\it type-A descent set} is defined to be 
	\[
	\Des_A(\pi,\epsilon):=\{i\in [n-1] \, : \, \pi_i^{c_i}>\pi_{i+1}^{c_{i+1}} \} 
	\]
and the {\it type-A descent statistic} is 
	\[
	\des_A(\pi,\epsilon):=\#\Des_A(\pi,\epsilon).
	\]
The {\it type-A major index} is 
	\[
	{\rm major}_A(\pi,\epsilon) :=\sum_{j\in\Des_A(\pi,\epsilon)} j
	\]		
\end{definition}
\begin{example}
Let $(\pi,\epsilon)=[2^1 \ 6^3 \ 4^3 \ 1^0 \ 5^2 \ 3^0]\in\Z_4 \wr S_6$. Then $\Des_A(\pi,\epsilon)=\{1,2,4\}$, $\des_A(\pi,\epsilon)=3$, and $\operatorname{major}_A(\pi,\epsilon)=7$.
\end{example}

We now review a different notion of descent statistics for $\Z_r\wr S_n$, namely the {\it negative} statistics. 

\begin{definition}
For an element $(\pi,\epsilon)\in\Z_r\wr S_n$, we define the {\it negative inverse multiset} as
	\[
	\NNeg(\pi,\epsilon):=\{\underbrace{i,i,\ldots,i}_{c_i \ {\rm times}} : i\in [n]\}.
	\]
	
The {\it negative descent multiset} is 
	\[
	\NDes(\pi,\epsilon):=\Des_A(\pi,\epsilon)\cup\NNeg((\pi,\epsilon)^{-1}).
	\]

The {\it negative descent statistic} is
	\[
	\ndes(\pi,\epsilon):=\# \NDes(\pi,\epsilon).
	\]

The {\it negative major index} is
	\[
	\nmajor(\pi, \epsilon):=\sum_{i\in \NDes(\pi,\epsilon)}i.
	\]
	
\end{definition} 
\begin{example}
If $(\pi,\epsilon)=[2^1 \ 6^3 \ 4^3 \ 1^0 \ 5^2 \ 3^0]\in\Z_4 \wr S_6$, then $(\pi,\epsilon)^{-1}=[4^0 \ 1^3 \ 6^0 \ 3^1 \ 5^2 \ 2^1]$ and hence $\NNeg((\pi,\epsilon)^{-1})=\{2,2,2,4,5,5,6\}$. Further,
	\[
	\NDes(\pi,\epsilon)=\{1,2,4\}\cup \{2,2,2,4,5,5,6\}=\{1,2,2,2,2,4,4,5,5,6\}
	\]
and thus $\ndes(\pi,\epsilon)=10$ and $\nmajor(\pi,\epsilon)=33$.
\end{example}

We will use the following representation for elements of $\Z_r\wr S_n$.

\begin{definition}
The \emph{increasing elements} of $\Z_r\wr S_n$, denoted $\mathcal{I}_{r,n}$, is the subset of elements satisfying $\des_A(\pi,\epsilon)=0$. 
\end{definition}  
It is a simple exercise to see that any element of $(\pi,\epsilon)\in\Z_r\wr S_n$ can be represented uniquely as 
	\[
	(\pi,\epsilon)=(\rho,\delta)\circ(\sigma,(1,1,\ldots,1))
	\]
for some $\sigma\in S_n$ and $(\rho,\delta)\in \mathcal{I}_{r,n}$. Subsequently, we have  that 
	\[
	\Z_r\wr S_n=\bigcup_{\sigma\in S_n}\mathcal{I}_{r,n}\sigma
	\]
where we use $\sigma$ in place of $(\sigma, (1,1,\ldots,1))$ for ease.

We also have the following observation.

\begin{proposition}{\cite[Proposition 5.11]{B-B}}\label{bijectionprop}
For $(\rho,\delta)\in \mathcal{I}_{r,n}$ and $\sigma\in S_n$, 
	\[
	\NNeg([(\rho,\delta)\sigma]^{-1})=\NNeg((\rho,\delta)^{-1}).
	\]
Further, each permutation $(\rho,\delta)\in \mathcal{I}_{r,n}$ is uniquely determined by $\NNeg((\rho,\delta)^{-1})$.	
\end{proposition}
\begin{remark}\label{notation}
We will often denote $(\pi,\epsilon)\in \Z_r\wr S_n$ by the pair $(\sigma, X)$ where $\sigma \in S_n$ satisfies $(\rho,\delta)\sigma=(\pi,\epsilon)$ with $(\rho,\delta)\in \mathcal{I}_{r,n}$ and $X=\NNeg((\pi,\epsilon)^{-1})$.
This establishes a bijective correspondence between elements of $\Z_r\wr S_n$ and pairs $(\sigma, X)$ with $\sigma \in S_n$ and $X$ a multiset of elements of $[n]$ in which each element appears with multiplicity strictly less than $r$. For convenience of notation, we will write $(\sigma,X)\in \Z_r\wr S_n$ when this interpretation is preferred.
\end{remark}

\begin{example}
Let $(\pi,\epsilon)=[2^1 \ 6^3 \ 4^3 \ 1^0 \ 5^2 \ 3^0]\in\Z_4 \wr S_6$ and consider $(\rho,\delta)=[4^3 \ 6^3 \ 5^2 \ 2^1 \ 1^0 \ 3^0]\in\mathcal{I}_{4,6}$ and $\sigma=421536\in S_6.$ Note that $(\pi,\epsilon)=(\rho,\delta)\circ (\sigma,(1,\ldots,1))$ as 
	\[
	[2^1 \ 6^3 \ 4^3 \ 1^0 \ 5^2 \ 3^0]=[4^3 \ 6^3 \ 5^2 \ 2^1 \ 1^0 \ 3^0] \circ [4^0 \ 2^0 \ 1^0 \ 5^0 \ 3^0 \ 6^0].
	\]
Moreover, $(\rho,\delta)^{-1}=[5^0 \ 4^3 \ 6^0 \ 1^1 \ 3^2 \ 2^1]$ and  $\NNeg((\rho,\delta)^{-1})=\{2,2,2,4,5,5,6\}.$ Therefore,\\ $\NNeg((\rho,\delta)^{-1})=\NNeg(((\rho,\delta)\sigma)^{-1})=\NNeg((\pi,\epsilon)^{-1})$. Thus, $(\pi,\epsilon)$ corresponds to the pair $(\sigma,X)=(421536,\{2,2,2,4,5,5,6\})$,
\end{example}

%%%%%%%%%%%%%%%%%%%%%%%%%%%%%%%%%%%%%%%%%%%%%%%%%%%%%%%%%%%

\section{$\Z_r\wr S_n$-quotient algebras of $R_n$ and descent bases}\label{quotients}

For convenience, we will view $R_n\cong T_n/I_n$ as the quotient of a polynomial ring by the toric ideal $I_n$. 
First consider the $S_n$ case. 
We define an $S_n$ action on $T_n$ given as $S_n\times T_n\to T_n$ defined on the variables by $(\pi, z_{A})\mapsto z_{\pi(A)}=z_{\{\pi(a_1),\ldots,\pi(a_k) \}}$ where $A=\{a_1,\ldots,a_k\}$. 
Note that this action passes to $R_n\cong T_n/I_n$, where it corresponds to the usual action of $S_n$ on $\{x_1,\ldots,x_n\}$ of permutation of the variables because
	\[
	z_{\{a_1,a_2,\ldots,a_k\}}\mapsto x_{a_1}x_{a_2}\cdots x_{a_k}
	\] 
and
	\[
	z_{\{\pi(a_1),\ldots,\pi(a_k) \}}\mapsto x_{\pi(a_1)}x_{\pi(a_2)}\cdots x_{\pi(a_k)}
	\]
which is the usual permutation of variables action of $S_n$.	
We consider the following ideal of elements which are invariant under this action:
	\[
	invar(1,n):=\left\langle \hat{e}_k:= \sum_{|A|=k}z_A \ | \ \mbox{for all} \ 0\leq k \leq n \right\rangle.
	\] 
The elements $\hat{e}_k$ are the $T_n$--analogue to the usual elementary symmetric functions $e_k$ in the polynomial ring on $n$ variables. 
Notice that this ideal cannot be the full ideal of invariants for this action on $T_n$ because there must be $2^n$ algebraically independent invariants \cite[Proposition 2.1.1]{Sturmfels}.
However, the generators are indeed invariant and this is the appropriate ideal for our purposes. 
We say the $S_n$ quotient algebra of $R_n$ is $R_n/\overline{invar(1,n)}$, where $\overline{invar(1,n)}$ is the image of $invar(1,n)$ in the quotient $T_n/I_n$. 
For convenience, we will consider the ring $T_n/J_{1,n}$ where $J_{1,n}:=invar(1,n)+I_n$, as it is a straightforward exercise in algebra to show that $T_n/J_{1,n}\cong R_n/\overline{invar(1,n)}$.

Next, we consider $\Z_r\wr S_n$ for $r\geq 2$. 
Consider the action $\Z_r\wr S_n \times T_n\to T_n$ defined on the variables by $\left((\pi,\epsilon), z_{A}\right)\mapsto (\prod_{i\in A}\epsilon_i) \cdot z_{\pi(A)}=(\prod_{i\in A}\epsilon_i) \cdot z_{\{\pi(a_1),\cdots,\pi(a_k) \}}$ where $A=\{a_1,\ldots,a_k\}$. We consider an ideal generated by invariant elements of this action:
	\[
	invar(r,n):=\left\langle z_\emptyset, \hat{e}_{r,k}:= \sum_{|A|=k}z_A^r \ | \ \mbox{for all} \ 1\leq k \leq n \right\rangle.
	\]
This is consistent with the above in the $r=1$ case. 
This ideal also does not contain all of the invariants of $T_n$ under this action, but the ideal is the appropriate choice of invariant generators for our scenario.
We say the $\Z_r\wr S_n$ quotient algebra of $R_n$ is $R_n/\overline{invar(r,n)}$, and we will consider the ring $T_n/J_{r,n}$ where $J_{r,n}:=invar(r,n)+I_n$, as we have $T_n/J_{r,n}\cong R_n/\overline{invar(r,n)}$.  

Now, we will define descent bases for our quotients.
First consider $T_n/J_{1,n}$. 
We wish to construct a basis based on the descent sets of $S_n$ that is analogous to the Garsia-Stanton descent basis. 
The Garsia-Stanton descent basis is a basis for the $S_n$--coinvariant algebra $\C[x_1,\ldots,x_n]/\mathscr{I}_n$ with coset representatives   
	\[
	a_\pi=\prod_{j\in \Des(\pi)}x_{\pi(1)}\cdots x_{\pi(j)}
	\]
for all $\pi\in S_n$, where the ideal $\mathscr{I}_n=\langle e_1(x_1,\ldots,x_n), \ldots ,e_n(x_1,\ldots,x_n)\rangle$ is generated the the elementary symmetric functions $e_i(x_1,\ldots,x_n)=\sum_{a_1<\cdots<a_i}x_{a_1}\cdots x_{a_i}$.
Garsia and Stanton originally showed this was a basis in \cite{GarsiaStanton} using the theory of Stanley-Reisner rings. 
In \cite{A-B-R}, Adin, Brenti, and Roichman provide another proof of this result and use the basis heavily in their proof of the Euler-Mahonian identity for $S_n$. 
We introduce an analogue of the Garsia-Stanton basis for $T_n/J_{1,n}$, which is
	\[
	\api:=\prod_{j\in\Des(\pi)} z_{\{\pi(1),\pi(2),\cdots 	\pi(j)\}}
	\]
for all $\pi\in S_n$. 
\begin{example}
Let $\pi=421536\in S_6$. Since $\Des(\pi)=\{1,2,4\}$, we have
	\[
	a_\pi=z_{\{4\}} z_{\{2,4\}}z_{\{1,2,4,5\}}
	\]
\end{example}

Because of the correspondence given in Theorem~\ref{isomorphism}, in this paper we will refer to the set $\{\api:\pi\in S_n\}$ as the \emph{Garsia-Stanton basis}.
Using Gr\"obner basis arguments in Section~\ref{groebner}, we will show that this is indeed a basis for $T_n/J_{1,n}$.

We can generalize this to a basis of $T_n/J_{r,n}$ for $r\geq 2$.
\begin{definition}
The \emph{negative descent basis} of $T_n/J_{r,n}$ consists of the elements
	\[
	b^r_{(\sigma,X)}:=\asigma\cdot \prod_{j\in X}z_{\{\sigma(1),\sigma(2),\cdots,\sigma(j)\}}
	\]
for all $\sigma\in S_n$ and $X$ a multiset of $[n]$ where no element has multiplicity greater than $r-1$. 
\end{definition}
\begin{example}\label{ex:3.3}
Let $(\sigma,X)=(421536,\{2,2,2,4,5,5,6\})$ corresponding to $(\pi,\epsilon)=[2^1 \ 6^3 \ 4^3 \ 1^0 \ 5^2 \ 3^0]\in\Z_4 \wr S_6$. Then

	\[
	\begin{array}{rcl}
	b_{(\sigma,X)} & = & \displaystyle z_{\{4\}} z_{\{2,4\}}z_{\{1,2,4,5\}} \cdot \left( z_{\{2,4\}}\right)^3\cdot\left( z_{\{1,2,4,5\}}\right)\cdot\left(z_{\{1,2,3,4,5\}} \right)^2\cdot \left(z_{\{1,2,3,4,5,6\}} \right)\medskip \\
	& = & \displaystyle z_{\{4\}} z_{\{2,4\}}^4 z_{\{1,2,4,5\}}^2 z_{\{1,2,3,4,5\}}^2  z_{\{1,2,3,4,5,6\}}\\
	\end{array}
	\]

\end{example}

We will show that this is a basis in Section~\ref{groebner}.
It follows from Remark~\ref{notation} that if $(\sigma,X)$ corresponds to $(\rho,\epsilon)\in \Z_r\wr S_n$, then $\NNeg((\rho,\epsilon)^{-1})=X$ and $\Des_A(\rho,\epsilon)=\Des(\sigma)$. 
So, elements of this basis correspond to $\NDes$ sets of $\Z_r \wr S_n$, hence the name ``negative descent basis.'' 
It is important to observe that this is distinct from the basis developed by R. Adin, F. Brenti, and Y. Roichman \cite{A-B-R} for the hyperoctohedral group $B_n\cong \Z_2\wr S_n$, as their basis related to the \emph{flag descent sets}.

%%%%%%%%%%%%%%%%%%%%%%%%%%%%%%%%%%%%%%%%%%%%%%%%%%%%%%%%%%%

\section{Descent bases via Gr\"obner bases for $J_{r,n}$}\label{groebner}

Our goal in this section is to prove the following theorem by finding a Gr\"obner basis for the ideal $J_{r,n}$.

\begin{theorem}\label{basistheorem}
For $r\geq 2$, $\{b^r_{(\sigma,X)}\, : \, (\sigma,X)\in\Z_r\wr S_n\}$ is a basis of $T_n/J_{r,n}$.
When $r=1$, $\{\api \, : \, \pi\in S_n\}$ is a basis of $T_n/J_{1,n}$.
\end{theorem} 

Before proving this theorem, we briefly review Gr\"obner bases. 
For a detailed reference on the theory and computation of Gr\"obner bases, we invite the reader to consult \cite{CoxLO,EneHerzog}.
 Consider  the polynomial ring $S=\C[x_1,\ldots,x_n]$.
Recall that a \emph{term order} $<_{\rm mon}$ on $S$ is a relation on $\Z_{\geq 0}^n$ which is a total ordering, a well ordering, and satisfies the condition that if $\alpha<_{\rm mon} \beta$ and $\gamma\in\Z_{\geq 0}^n$, then $\alpha+\gamma <_{\rm mon} \beta+\gamma$. Given two monomials $m_1=\prod_{i=1}^n x_i^{\alpha_i}$ and $m_2=\prod_{i=1}^n x_i^{\beta_i}$, we say that $m_1<_{\rm mon} m_2$ if $(\alpha_1,\alpha_2,\ldots,\alpha)<_{\rm mon} (\beta_1,\beta_2,\ldots,\beta_n)$. 
Given $f\in S$, the \emph{leading monomial of} $f$, denoted $\operatorname{LM}(f)$, is the largest monomial of $f$ with respect to the term order  $<_{\rm mon}$.
For notation, we will denote monomials as $\x^\alpha:=\prod_{i=1}^nx_i^{\alpha_i}$ where $\alpha\in\Z_{\geq 0}^n$. 
The \emph{leading term} of $f$, denoted $\LT(f)$, is the leading monomial with its coefficient. 
Given an ideal $I\subset S$, $\LT(I)=\{c\x^a \, : \,\exists \ f\in I \ s.t. \ \LT(f)=c\x^a\}$ and $\langle \LT(I)\rangle$ is the ideal generated by elements of $\LT(I)$, which we call the \emph{leading term ideal} of $I$. 
A finite subset $G=\{g_1,\ldots,g_t\}$ of an ideal $I$ is called a \emph{Gr\"obner basis} for $I$ if
	\[
	\langle\LT(g_1), \ldots,\LT(g_t)\rangle =\langle\LT(I)\rangle.
	\]
Given a polynomial ideal $I\subset S$ and a fixed term order, we can algorithmically construct a Gr\"obner basis using a classical result known as \emph{Buchberger's Algorithm}. 
However, one can optimize this classical algorithm to be more efficient. 
Before stating an optimized version, we must introduce notation. Given two polynomials $f,g$, the \emph{$S$-polynomial} of $f$ and $g$ is 
	\[
	S(f,g)=\frac{\x^\gamma}{\LT(f)}\cdot f-\frac{\x^\gamma}{\LT(g)}\cdot g
	\]
where $\x^\gamma=\operatorname{lcm}(\operatorname{LM}(f),\operatorname{LM}(g))$. Given a polynomial $f$ and an ordered $s$-tuple of polynomials $F=(f_1,\cdots,f_s)$, let $\overline{f}^F$ denote the reminder of $f$ after division by each polynomial in $F$ performed in order. The reader should consult \cite{CoxLO} for a thorough discussion of multivariate polynomial division.	
   
\begin{algorithm}[Optimized Buchberger Algorithm]
Let $I=\langle f_1,\ldots,f_s\rangle\subset S$. 
Then a Gr\"obner basis for $I$ can be constructed in a finite number of steps as follows:\\
\noindent Input: $F=(f_1,\ldots,f_s)$\\
\noindent Output: $G$, a Gr\"obner basis for $I$\\
\indent Initial state: $B:=\{(i,j)\, :\, 1\leq i<j\leq s\}$; \ \ $G:=F$; \ \ $t:=s$\smallskip

\indent \textbf{WHILE} $B\neq \emptyset$ \textbf{DO}\\
\indent\indent\indent Select $(i,j)\in B$\\
\indent\indent\indent \textbf{IF} $\operatorname{lcm}(\LT(f_i),\LT(f_j))\neq \LT(f_i)\cdot\LT(f_j)$, \textbf{AND} $\operatorname{Criterion}(f_i,f_j,B)$ is false \textbf{THEN}\\
\indent\indent\indent\indent\indent $\mathcal{S}:=\overline{S(f_i,f_j)}^G$\\
\indent\indent\indent\indent\indent \textbf{IF} $\mathcal{S}\neq 0$ \textbf{THEN}\\
\indent\indent\indent\indent\indent\indent\indent $t:=t+1$ \ ; \ $f_t:=\mathcal{S}$ \ ; \ $G:=G\cup \{f_t\}$ \ ; \ $B:=B\cup\{(i,t) \, : \, 1\leq i\leq t-1\}$\\
\indent\indent\indent $B:=B-\{(i,j)\}$,\smallskip

\noindent where $\operatorname{Criterion}(f_i,f_j,B)$ is true provided that there is some $k\not\in\{i,j\}$ for which the pairs $[i,k]$ (i.e $(i,k)$ if $i<k$ or $(k,i)$ if $k<i$) and $[j,k]$ are NOT in $B$ and $\LT(f_k)$ divides $\operatorname{lcm}(\LT(f_i),\LT(f_j))$.
\end{algorithm}

Our motivation to compute a Gr\"obner basis for the ideal $J_{r,n}$ is the following theorem attributed to Macaulay.
\begin{theorem}[Macaulay, c.f.\cite{EneHerzog}]\label{Mac}
Let $<$ be a term order and let $I\subset S$ be an ideal. 
Then the monomials in $S$ which do not belong to $\langle\LT(I)\rangle$ form a $\C$-basis for $S/I$. 
\end{theorem}

Determining a Gr\"obner basis for $J_{r,n}$ yields a useful description of $\langle\LT(J_{r,n})\rangle$. 
Thus, Theorem~\ref{basistheorem} is an immediate consequence of Theorem~\ref{Mac}, Proposition~\ref{LT}, and Theorem~\ref{gb} below.
 
\begin{proposition}\label{LT}
Fix $r\geq 1$ and $n\geq 1$.
Consider the monomial ideal $N_{r,n}$ in $T_n$ generated by the following elements:
	\begin{itemize}
	\item  $z_\emptyset$	
	\item $z_A^r$, where $A=[k]$ for all $1\leq k\leq n$
	\item $z_A^{r+1}$ where $A\neq [k]$ for any $0\leq k\leq n$
	\item $z_Az_B$  such that $A\not\subseteq B$ and $B\not\subseteq A$
	\item $z_A^rz_B$ where $A\neq [k]$ for any $0\leq k\leq n$, such that $A\subset B$ and $\min(B\setminus A)>\max(A)$
	\item $z_A z_B^r$ where $B\neq [k]$ for any $0\leq k\leq n$, such that $A\subset B$  and there is an $\ell$ with  $[\ell]\not\subset A$, $[\ell]\subset B$, and $B\setminus A\subset [\ell]$
	\item $z_{A_1}z_{A_{2}}^r z_{A_3}$, where $A_2\neq[k]$ for any $0\leq k\leq n$, such that $A_1\subset  A_{2}\subset A_3$ and $\max(A_{2}\setminus A_{1})<\min(A_3\setminus A_{2})$
	\end{itemize}
The monomials outside of this ideal are precisely the elements of the negative descent basis for $T_n/J_{r,n}$ (for $r=1$, this is the Garsia-Stanton basis described above).
\end{proposition}

\begin{proof}
We will first show the argument for $r=1$, the Garsia-Stanton basis, then we will generalize the argument for $r\geq 2$. 
Assume unless otherwise stated that elements of sets are written in ascending order, e.g. $A=\{a_1,a_2,\ldots,a_\ell\}$ implies $a_1<a_2<\cdots<a_\ell$. 
First, note that the following observations imply that every monomial $\api$ is not divisible by any of the generators of $N_{1,n}$.
	\begin{itemize}
	\item $z_{\emptyset}$ clearly cannot divide $\api$ by construction.
	\item $z_{\{1,2,3,\ldots,k\}}$ cannot divide $\api$, as this would imply that there is a descent at the position $k$, but there is no element smaller than $k$ which has not already appeared.
	\item $z_A^2$ cannot divide $\api$, as by definition each set $A$ which arises from $\Des(\pi)$ must be unique. 
	\item By definition, if $z_Az_B$ is a factor of $\api$, it implies that $A\subset B$ or vice versa. 
	\item If $z_Az_B$ divides $\api$ with $A\subset B$ such that $A=\{a_1,a_2,\dots a_{\ell}\}$ and $B=A \cup \{b_1,\dots,b_k\}$, we must have that $b_1<a_{\ell}$ else there is no descent possible at position $\ell$. 
	\item If $z_Az_B$ divides $\api$ with $A\subset B$ such that $[\ell]\not\subset A$ and $[\ell]\subset B$, we must have some element  $x\in B\setminus A$ such that $x\not\in [\ell]$, else no descent could occur since $\pi(|B|)\in [\ell]$ and $[\ell]\subset\{\pi(1),\ldots,\pi(|B|)\}$.
	\item If $z_{A_1}z_{A_{2}}z_{A_3}$ divides $\api$ where $A_1\subset  A_{2}\subset A_3$ and $\max(A_{2}\setminus A_{1})<\min(A_3\setminus A_{2})$, then no descent could occur between set $A_2$ and $A_3$, i.e. in position $\pi(|A_2|)$.
	\end{itemize}
	
Suppose next that we have a monomial in $m\in T_n$, which is divisible by none of the generators of $N_{1,n}$. 
We claim that there exists some $\pi\in S_n$ such that $m=\api$; to prove this claim, first we write 
	\[
	m=z_{B_1}z_{B_2}\cdots z_{B_s}
	\]  
where $B_1\subset B_2 \subset \cdots \subset B_s$.
We denote $B_1=\{\beta_{1_1},\ldots, \beta_{1_{m_1}}\}$ and $B_i=B_{i-1}\cup \{\beta_{i_1},\ldots,\beta_{i_{m_i}}\}$ for all $1<i\leq s$.
Note that this union corresponds to the permutation  
	\[
	\pi=\beta_{1_1}\cdots \beta_{1_{m_1}} \beta_{2_1}\cdots \beta_{2_{m_2}}\cdots\cdots \beta_{s_1}\cdots \beta_{s_{m_s}} \gamma_1\cdots\gamma_t
	\]
where $\gamma_1<\gamma_2<\cdots<\gamma_t$ are the elements which do not appear in any $B_i$ set. 
Moreover, we have that  $\beta_{i_{m_i}}>\beta_{{i+1}_1}$ and $\beta_{s_{m_s}}>\gamma_1$ and  these will be the only such descents since $m$ is not divisible by any of the generators of $N_{r,n}$.
Hence, $m$ is a Garsia-Stanton descent element $\api$. 
(This argument is similar to standard $P$-partition arguments \cite[Lemma 3.15.3]{EC1}.)

Now suppose that $r\geq 2$.
By a similar argument to that just given, $b_{(\pi,X)}^r$ is not divisible by a monomial from among the generators of $N_{r,n}$, since:
	\begin{itemize}
	\item $z_{\emptyset}$ clearly cannot divide $b_{(\pi,X)}^r$ by construction.	
	\item $z_{[k]}^r$ cannot appear in $b_{(\pi,X)}^r$, as, since the greatest possible multiplicity of any element in $X$ is $r-1$, this would imply that there is a position $k$ descent in $\pi$ when all smaller elements than $\pi(k)$ have already appeared in $\pi$. 
	\item $z_A^{r+1}$ for $A\neq[k]$ cannot appear in $b_{(\pi,X)}^r$ as we only obtain a single $z_A$ from $\api$, and we can obtain at most $r-1$ copies of $z_A$ from the product over $X$.
Note that if $z_A^r$ appears in $b_{(\pi,X)}^r$, then one of the $z_A$ terms must have come from the product indexed by $\Des(\pi)$, and thus $|A|\in \Des(\pi)$.
	\item By definition, $z_Az_B$ a factor of $b_{(\pi,X)}^r$ implies $A\subseteq B$ or vice-versa.
	\item If $z_A^r z_B$ appears in $b_{(\pi,X)}^r$ where $A\subset B$ with $A=\{a_1,a_2,\dots a_{\ell}\}$ and $B=A \cup \{b_1,\dots,b_k\}$, it follows that $|A|\in\Des(\pi)$, thus we must have that $b_1<a_{\ell}$ else there is no descent occurring in $\pi$ in position $|A|$.
	\item If $z_Az_B^r$ appears in $b_{(\pi,X)}^r$ where $A\subset B$ with $[\ell]\not\subset A$ and $[\ell]\subset B$, then $|B|\in\Des(\pi)$. 
Hence, there must exist an element $x\in B\setminus A$ such that $x\not\in [\ell]$, else no descent can occur.
	\item If $z_{A_1}z_{A_{2}}^rz_{A_3}$ appears in $b_{(\pi,X)}^r$ such that $A_1\subset A_2 \subset A_3$ and $\max(A_{2}\setminus A_{1})<\min(A_3\setminus A_{2})$, then no descent can occur between set $A_2$ and $A_3$, i.e. in position $\pi(|A_2|)$, but the power of $r$ on $z_{A_{2}}^r$ forces that there is such a descent. Hence this divisibility is not possible.
	\end{itemize}	  

Suppose next that we have a monomial $m_r\in T_n$ that is divisible by none of the generators of $N_{r,n}$.
We claim that there exists some $\pi\in S_n$ and $X$ a multiset of $[n]$ with every element having multiplicity strictly less than $r$ such that $m_r=b_{(\pi,X)}^r$.
An example illustrating the following proof is given in Example~\ref{ex:Xtilde}.
To prove this claim, first we write 
	\[
	m_r=z_{B_1}^{b_1}z_{B_2}^{b_2}\cdots z_{B_s}^{b_s}
	\]
where we have $B_1\subset B_2 \subset\cdots\subset B_s$.
Note that $b_i\leq r$ if $B_i\neq [k]$ and $b_i\leq r-1$ if $B_i=[k]$.
As in the previous case, inductively define $B_i=B_{i-1}\cup \{\beta_{i_1},\ldots,\beta_{i_{m_i}}\}$.
Construct a new monomial
	\[
	m_r'=\left\{\begin{array}{ll}z_{B_1}z_{B_2}\cdots z_{B_s} & \text{if } B_s\neq [n] \\ z_{B_1}z_{B_2}\cdots z_{B_{s-1}} & \text{if } B_s=[n] \end{array}\right.
	\]  
and the set 
	\[
	\tilde{X}=\left\{ \begin{array}{ll}\{\underbrace{c_1,\cdots, c_1}_{b_1-1 \ {\rm times}}, \underbrace{c_2,\cdots, c_2}_{b_2-1 \ {\rm times}},\cdots, \underbrace{c_s,\cdots,c_s}_{b_s-1 \ {\rm times}}\}  & \text{if } B_s\neq [n] \\
\{\underbrace{c_1,\cdots, c_1}_{b_1-1 \ {\rm times}}, \underbrace{c_2,\cdots, c_2}_{b_2-1 \ {\rm times}},\cdots, \underbrace{c_{s-1},\cdots,c_{s-1}}_{b_{s-1}-1 \ {\rm times}}, \underbrace{c_s,\cdots,c_s}_{b_s \ {\rm times}}\} & \text{if } B_s=[n]  \end{array}
 \right.
	\]
where $c_i=|B_i|$.
We associate to $m_r'$ the permutation
	\[
	\pi=\beta_{1_1}\cdots \beta_{1_{m_1}} \beta_{2_1}\cdots \beta_{2_{m_2}}\cdots\cdots \beta_{s_1}\cdots \beta_{s_{m_s}} \gamma_1\cdots\gamma_t
	\]
where $\gamma_1<\gamma_2<\cdots<\gamma_t$ are any elements which do not appear in any $B_i$ set. 
Since the $\beta$-values within each $B_i$ are increasing, the only possible descents occur between $\beta_{i_{m_i}}$ and $\beta_{{i+1}_1}$. 
If we have $\beta_{i_{m_i}} > \beta_{{i+1}_1}$, then we have a descent and we do nothing. 
(Note that the final three types of generators of $N_{r,n}$ force a descent to occur if $b_i$ takes on a maximal value of $r$ or $r-1$, showing that all seven of the types of generators of $N_{r,n}$ are required for this argument to hold.)
If we have 
\begin{equation}\label{eqn:nodes}
\beta_{i_{m_i}} < \beta_{{i+1}_1}
\end{equation}
 then there is no descent.
Let $m_{fail}$ be the product of $z_{B_i}$ over all the $i$ values such that~\eqref{eqn:nodes} holds and define
	\[
	\widetilde{m_r}:=m_r'/m_{fail} \, .
	\]
We have that $\api=\widetilde{m_r}$ by our argument in the $r=1$ case. 
Moreover, we set $X:=\tilde{X}\cup \{c_i: {\rm where } \ \beta_{i_{m_i}} < \beta_{{i+1}_1}\}$, where as before $c_i=|B_i|$. 
With this choice of permutation and multiset we obtain $m=b^r_{(\pi,X)}$.	
\end{proof}

\begin{example}\label{ex:Xtilde}
Let $r=4$ and $n=6$ and use the notation from the preceding proof.
Consider the monomial \
\[
z_{\{4\}} z_{\{2,4\}}^4 z_{\{1,2,4,5\}}^2 z_{\{1,2,3,4,5\}}^2  z_{\{1,2,3,4,5,6\}} \, .
\]
Thus, $B_2=\{2,4\}$, $B_5=\{1,2,3,4,5,6\}$, and so on.
We have that $m_4'=z_{\{4\}} z_{\{2,4\}} z_{\{1,2,4,5\}} z_{\{1,2,3,4,5\}}$ since $B_5=\{1,2,3,4,5,6\}$.
We have that $\widetilde{X}=\{2,2,2,4,5,6\}$, where the $6$ is included since $B_5=\{1,2,3,4,5,6\}$.
In this case, the permutation $\pi=421536$, and $m_{fail}=z_{\{1,2,3,4,5\}}$.
Thus, we have that
\[
\widetilde{m_4}=z_{\{4\}} z_{\{2,4\}} z_{\{1,2,4,5\}} = \api
\]
and
\[
X=\widetilde{X}\cup \{5\} = \{2,2,2,4,5,5,6\} \, .
\]
It is straightforward to check that 
\[
b^4_{(421536,\{2,2,2,4,5,5,6\})}=z_{\{4\}} z_{\{2,4\}}^4 z_{\{1,2,4,5\}}^2 z_{\{1,2,3,4,5\}}^2  z_{\{1,2,3,4,5,6\}} 
\]
as desired.
Note that here we have recovered the correspondence given in Example~\ref{ex:3.3}.
\end{example}

\begin{definition}
Given two sets $A$ and $B$ such that $|A|=|B|=k$, we say that $A$ is \emph{lexicographically before} $B$ if there exists $i\in A$ such that $i\not\in B$ and given any $j\in B$ such that $j<i$ we have $j\in A$.
\end{definition}
For example, the  ordering of $3$-subsets of the $5$-set would be 
${1,2,3}<{1,2,4}<{1,2,5}<{1,3,4}<{1,3,5}<{1,4,5}<{2,3,4}<{2,3,5}<{2,4,5}<{3,4,5}$.
Our next step is to prove that the monomials listed in Proposition~\ref{LT} arise as leading terms of $J_{r,n}$ when the following monomial term order is imposed on $T_n$.

\begin{definition}
Give the variables of $T_n$ the linear order $z_A>z_B$ if $|A|<|B|$ or if $|A|=|B|$ and $A$ is lexicographically before $B$.
With respect to this ordering of variables, endow $T_n$ with the graded reverse lexicographic (or \emph{grevlex}) term order.
In this setting, grevlex order is as follows.
Let $(\alpha_A)_{A\subseteq [n]}$ and $(\beta_A)_{A\subseteq [n]}$ be vectors in $\Z_{\geq 0}^{2^n}$ with entries totally ordered by setting the $A$-th coordinate to be larger than the $B$-th coordinate if and only if $z_A>z_B$.
For two monomials in $T_n$, we have 
\[
\prod_{A\subseteq [n]}z_A^{\alpha_A} >_{grevlex} \prod_{A\subseteq [n]}z_A^{\beta_A}
\]
 if either (1) $\sum_{A\subseteq [n]}\alpha_A > \sum_{A\subseteq [n]}\beta_A$ or (2) $\sum_{A\subseteq [n]}\alpha_A = \sum_{A\subseteq [n]}\beta_A$ and in $(\alpha_A-\beta_A)_{A\subseteq [n]}$ the right most non-zero entry is negative.
\end{definition}

\begin{example}
The variables in $T_3$ are ordered as follows:
\[
z_\emptyset > z_{\{1\}} > z_{\{2\}} > z_{\{3\}} > z_{\{1,2\}} > z_{\{1,3\}} > z_{\{2,3\}} > z_{\{1,2,3\}}
\]
We have that
\[
z_{\{2\}}^4>z_\emptyset z_{\{1\}}^2 z_{\{1,2\}}
\]
since the exponent vectors for these monomials with respect to the linear order of the variables above are $(0,0,4,0,0,0,0,0)$ and $(1,2,0,0,1,0,0,0)$, hence we have
\[
(0,0,4,0,0,0,0,0)-(1,2,0,0,1,0,0,0)=(-1,-2,4,0,-1,0,0,0)
\]
with negative right-most non-zero entry.
\end{example}

We will need the following definition for the proof of Theorem~\ref{gb}.
\begin{definition}
We call a pair of subsets $A$ and $B$ such that $A\nsubseteq B$ and $B\nsubseteq A$ a \emph{Sperner 2-pair}.
\end{definition}

\begin{theorem}\label{gb}
There exists a Gr\"{o}bner basis $G_{r,n}$ of $J_{r,n}$ for which the ideal generated by $\LT(G_{r,n})$ is the ideal $N_{r,n}$ generated by terms of the form listed in Proposition~\ref{LT}. 
\end{theorem}

Prior to proving the general Gr\"obner basis result, it is useful to consider a small example.
Take $J_{3,2}=\langle z_{\{1\}}z_{\{2\}}-z_\emptyset z_{[2]}, z_\emptyset, z_{\{1\}}^3+z_{\{2\}}^3,z_{[2]}^3\rangle$. From the list of desired leading terms given in Proposition \ref{LT}, the only term not immediately accounted for is $z_{\{2\}}^3$. The only nontrivial $S$-polynomial to consider initially is
	\begin{align*}
	S(z_{\{1\}}z_{\{2\}}-z_\emptyset z_{[2]}, z_{\{1\}}^3+z_{\{2\}}^3) & =\frac{z_{\{1\}}^3z_{\{2\}}}{z_{\{1\}}z_{\{2\}}}\cdot\left(z_{\{1\}}z_{\{2\}}-z_\emptyset z_{[2]}\right)-\frac{z_{\{1\}}^3z_{\{2\}}}{z_{\{1\}}^3}\cdot \left(z_{\{1\}}^3+z_{\{2\}}^3\right)\\
	 & =-z_{\{2\}}^4-z_\emptyset z_{\{1\}}^2 z_{\{1,2\}}.
	\end{align*}
Under our term order, the leading term is $-z_{\{2\}}^4$, which is as desired. 
In order to show that no additional polynomials appear in the Gr\"obner basis, an exhaustive check of all other $S$-polynomials shows they reduce to 0. 
Alternatively, we can argue that no other terms will appear because we can compute that $\dim_\C(T_2/J_{3,2})=3^2\cdot 2=18$ via a Hilbert series argument that is explicitly given by \eqref{HilbDegreeArguement} in the proof below, thus no other leading terms can appear without contradicting this known dimension. 
In small examples, either argument will suffice. However, for arbitrary $r$ and $n$, the latter argument is more efficient. 

\begin{proof}[Proof of Theorem~\ref{gb}]
Use the term order for $T_n$ described above.
Our proof will involve computing $S$-polynomials starting from the generators of $J_{r,n}$.
To minimize the number of computations required, we first make a dimension argument showing that the number of monomials outside of the leading term ideal for $J_{r,n}$ is the number of elements of the negative descent basis.
We then compute $S$-polynomials to produce elements with all of the leading terms listed in Proposition~\ref{LT}, which will complete the proof. 
We will compute the $S$-polynomials for arbitary $r$, but we will make two dimension arguments, for $r=1$ and $r\geq 2$. 

Consider $r=1$. 
It is a straightforward observation to notice that the number of elements of $p\in R_n$ such that $\deg(p)=t^k$ are precisely the lattice points at height $k$ in the $\cone([0,1]^n)$ and the cardinality of these elements is $(k+1)^n$.  
Combining this observation with \cite[Proposition 1.4.4]{EC1}, we see that the Hilbert series of $R_n$ is given by 
	\[
	\Hilb{R_n}{t}= \sum_{k\geq 0}(k+1)^n t^k=\frac{A_n(t)}{(1-t)^{n+1}}
	\]
where $A(n)=\sum_{\pi\in S_n}t^{\des(\pi)}$ is the Eulerian polynomial. 
Let $\mathcal{C}_{1,n}:=\C[\hat{e}_k+I_n|0\leq k\leq n]$, and note that the elements $\hat{e}_k+I_n$ are algebraically independent since they specialize in $R_n$ (by setting $t=1$) to the usual elementary symmetric functions; note that $\Hilb{\mathcal{C}_{1,n}}{t}=\frac{1}{(1-t)^{n+1}}$.
Hochster's Theorem implies that $R_n$ is Cohen-Macaulay \cite{Hochster}, and since $invar(1,n)$ is an ideal generated by an algebraically independent system of parameters, we have 
	\[
	\Hilb{T_n/J_{1,n}}{t}=A_n(t)
	\] 	  
by \cite[Lemma 17.1]{Hibi}. 
The $\C$--dimension of $T_n/J_{1,n}$ is 
\[
\dim_{\C}(T_n/J_{1,n})=\Hilb{T_n/J_{1,n}}{1}=A(1)=n! \, ,
\]
 which is the number of elements in the Garsia-Stanton descent basis, as desired.	

Now, suppose that $r\geq 2$. 
Let $\mathcal{C}_{r,n}=\C[z_\emptyset+I_n,\hat{e}_{r,k}+I_n | 1\leq k\leq n]$. 
Given that $R_n$ is Cohen-Macaulay and that $\hat{e}_{r,k}+I_n$ and $z_\emptyset+I_n$ are algebraically independent, hence $\Hilb{\mathcal{C}_{r,n}}{t}=\frac{1}{(1-t)(1-t^r)^n}$, we have that
	\[
	\Hilb{R_n}{t}= \sum_{k\geq 0}(k+1)^n t^k=\frac{B_{r,n}(t)}{(1-t)(1-t^r)^{n}}
	\]
where $B_{r,n}(t)=A_n(t)\cdot (1+t+\cdots+t^{r-1})^n$ by our previous calculation for $r=1$.
Thus, 
\[
	\Hilb{T_n/J_{r,n}}{t}=A_n(t)\cdot (1+t+\cdots+t^{r-1})^n
\]
from which we can conclude that 
\begin{equation}\label{HilbDegreeArguement}
\dim_{\C}(T_n/J_{r,n})=\Hilb{T_n/J_{r,n}}{1}=B_{r,n}(1)=r^nn! \, ,
\end{equation}
which is the number of elements in the negative descent basis, as desired.

Next, we move to $S$-polynomial calculations.
Our goal is to compute $S$-polynomials until all the elements listed in Proposition~\ref{LT} arise as leading terms; since at that point we will have reached the correct value of $\dim_\C (T_n/J_{r,n})=\dim_\C (T_n/\LT(J_{r,n}))$, we must have a Gr\"obner basis.

We begin by noting that some of our desired leading terms arise from the generators of $J_{r,n}$.
First, $z_Az_B$ such that $A\not\subset B$ and $B\not\subset A$ where $A\neq[k]\neq B$ for any $k$ are leading terms of $I_n$. 
%Such sets form the leading terms for elements from $I_n$, which will have no nontrivial $S$-polynomials with elements from $invar(r,n)$. 
The monomials $z_\emptyset$ and $z_A^r$ where $A=[k]$ for $k=1,\ldots,n$ are the leading terms of $invar(r,n)$.
These account for the fourth, first, and second items listed in Proposition~\ref{LT}, respectively. 

To obtain an element with the leading term $z_A^{r+1}$ as given in the third bullet of Proposition~\ref{LT}, suppose that $|A|=k$ and consider the following $S$-polynomial:
	
\begin{align*}
  & S(\hat{e}_{r,k},z_{[k]}z_A-z_{[k]\cap A}z_{[k]\cup A} ) \\
= & \frac{z_{[k]}^rz_A}{z_{[k]}^r}\left(z_{[k]}^r+z_{A_1}^r+z_{A_2}^r+\cdots +z_{A}^r+\cdots + z_{A_{{n \choose k}-1}}^r \right) \\
  & \hspace{45mm} -\frac{z_{[k]}^rz_A}{z_{[k]}z_A}(z_{[k]}z_A-z_{[k]\cap A}z_{[k]\cup A}) \\
= & z_A\left(z_{A_1}^r+z_{A_2}^r+\cdots +z_{A}^r+\cdots + z_{A_{{n \choose k}-1}}^r \right)+z_{[k]}^{r-1}z_{[k]\cap A}z_{[k]\cup A}
\end{align*}

Note that the term order implies that 
\[
z_Az_{A_1}^r>z_Az_{A_2}^r>\cdots>z_Az_{A}^r>\cdots z_Az_{A_{{n \choose k}-1}}^r>z_{[k]}^{r-1}z_{[k]\cap A}z_{[k]\cup A}
\]
However, for each $i$ where $A_i\neq A$, $z_Az_{A_i}$ is the leading term of a polynomial of $J_{r,n}$, and we use $z_Az_{A_i}-z_{A\cap A_i}z_{A\cup A_i}\in J_{r,n}$ to rewrite $z_Az_{A_i}^r$, yielding 
	\begin{equation}
	S(\hat{e}_{r,k},z_{[k]}z_A-z_{[k]\cap A}z_{[k]\cup A} )=z_A^{r+1}+\sum_j z_{A\cap A_j}z_{A_j}^{r-1} z_{A\cup A_j}\label{term}
	\end{equation}
where the sum is over all $j$ such that $|A_j|=k$, $A_j\neq A$, and $A\cap A_j\neq \emptyset$, since any terms involving $z_\emptyset$ are elements of $J_{r,n}$.
The observation that $|A|<|A\cup A_j|$ for all such $j$ implies that $z_A^{r+1}$ is the leading term of this polynomial, as desired.
	
Assume that we have added all prior $S$-polynomial calculations to the generators of $J_{r,n}$.
To obtain terms of the form $z_A^r z_B$, where $A\subset B$ with $\max(A)<\min(B \setminus A)$ as listed in the fifth bullet of Proposition~\ref{LT}, let $|A|=k$.
We compute the $S$-polynomial of $\hat{e}_{r,k}$ and the generator of $I_n$ with leading term $z_{[k]}z_{B}$. 
Note that $z_{[k]}z_{B}$ is the leading term of a generator of $I_n$, since by assumption $A\neq [k]$ thus if $[k]\subset B$ this would violate the condition $\max(A)<\min(B \setminus A)$. 
We compute:
	
	\begin{align*}
	& S(\hat{e}_{r,k},z_{[k]}z_{B}-z_{[k]\cap B}z_{[k]\cup B})\\  = &  \frac{z_{[k]}^r z_B}{z_{[k]}^r }\left(z_{[k]}^r +z_{A_1}^r+z_{A_2}^r+\cdots +z_{A}^r+\cdots + z_{A_{{n \choose k}-1}}^r \right)\\
	 & \hspace{45mm} -\frac{z_{[k]}^r z_B}{z_{[k]}z_B}\left( z_{[k]}z_{B}-z_{[k]\cap B}z_{[k]\cup B}\right)\\
	 = & z_B\left(z_{A_1}^r+z_{A_2}^r+\cdots +z_{A}^r+\cdots + z_{A_{{n \choose k}-1}}^r \right)+z_{[k]}^{r-1}z_{[k]\cap B}z_{[k]\cup B}
	\end{align*}
We have the ordering
	\[
	z_{A_1}^r z_B>z_{A_2}^r z_B>\cdots >z_A^r z_B>\cdots >z_{A_{{n\choose k}-1}}^r z_B>z_{[k]}^{r-1}z_{[k]\cap B}z_{[k]\cup B} \, .
	\]
Moreover, by the condition $\max(A)<\min(B \setminus A)$ and the use of lexicographic order on subsets, we know that ${A_i}\not\subset B$ for all $i$ such that $z_{A_i}^r z_B>z_A^r z_B$.
This is true because if $A_i\subset B$, then there exists some $j\in A_i$, $j\not\in A$ so that $\max(A)<j$ and the condition that $|A|=|A_i|$ implies that there must exist some $s\in A$ such that $s\not\in A_i$ and for all $t\in A_i$ such that $t<s$ we have $t\in A$, which would contradict  $z_{A_i}^r z_B>z_A^r z_B$ by the definition of our variable ordering arising from the lexicographic ordering on subsets.
The condition that ${A_i}\not\subset B$ implies that $z_{A_i}^r z_B$ is a leading term of a polynomial in $I_n$.
Applying $z_{A_i}z_B-z_{A_i\cap B}z_{A_i\cup B}\in J_{r,n}$ to the term $z_{A_i}^rz_B$ will produce $z_{A_i\cap B}z_{A_i}^{r-1}z_{A_i\cup B}<z_A^r z_B$. 
Therefore, we will have
	\begin{align*}
	&S(\hat{e}_{r,k},z_{[k]}z_B-z_{[k]\cap B}z_{[k]\cup B}) = \\
        &z_A^rz_B+\sum_jz_{A_j}^r z_B+\sum_m z_{A_m\cap B}z_{A_m}^{r-1}z_{A_m\cup B}
	\end{align*}
where the first sum is over all $j$ so that $|A_j|=|A|$, $A_j\neq A$, and $A_j\subset B$, which implies that $z_{A}>z_{A_j}$ by condition $\max(A)<\min(B \setminus A)$. 
The second sum is over all $m$ such that $|A_m|=|A|$ where $A_m$ and $B$ are a Sperner 2-pair with $A_m\cap B\neq \emptyset$, as if the intersection was empty then the resulting term would be a multiple of $z_\emptyset$ and hence an element of $J_{r,n}$.
It follows from a simple cardinality argument that $z_{A_m\cup B}<z_B$, and thus $z_A^rz_B$ is a leading term in $J_{r,n}$.

Assume again that we have added all prior $S$-polynomial calculations to the generators of $J_{r,n}$.
To obtain terms of the form $z_Az_B^r$ where there is an $\ell$ such that $[\ell]\not\subset A$, $[\ell]\subset B$ and $B\setminus A\subset [\ell]$, as listed in the sixth bullet of Proposition~\ref{LT}, let $|B|=k$.
We compute the $S$-polynomial of $\hat{e}_{r,k}$ and the generator of $I_n$ with leading term $z_Az_{[k]}$, which is a leading term since there exists an element $x\in[\ell]\subset[k]$ such that $x\notin A$ and there also exists $y=\max(A)=\max(B)\notin[k]$:
	\[
	\begin{array}{rcl}
	& & S(\hat{e}_{r,k},z_Az_{[k]}-z_{A\cap[k]}z_{A\cup[k]}) \\
        & = & \frac{z_Az_{[k]}^r}{z_{[k]}^r}\left( z_{[k]}^r+z_{B_1}^r+\cdots +z_{B}^r+\cdots+z_{B_{{n\choose k}-1}}^r \right)\\
	& & -\frac{z_Az_{[k]}^r}{z_Az_{[k]}}\left(z_Az_{[k]}-z_{A\cap[k]}z_{A\cup[k]} \right)\\
	& = & z_A \left(z_{B_1}^r+\cdots +z_{B}^r+\cdots+z_{B_{{n\choose k}-1}}^r \right) +z_{[k]}^{r-1}z_{A\cap[k]}z_{A\cup[k]}
	\end{array}
	\]	
which yields the term order of 
	\[
	z_Az_{B_1}^r>z_Az_{B_2}^r>\cdots>z_Az_B^r>\cdots z_Az_{B_{{n\choose k}-1}}^r>z_{[k]}^{r-1}z_{A\cap[k]}z_{A\cup[k]} \, .
	\]
Note that $A\not\subset B_i$ for all $i$ such that $z_{B_i}>z_B$. 
This is true because if $A\subset B_i$ for $B_i\neq B$, then given that $|B|=|B_i|$ we must have $z_{B}>z_{B_i}$ because $B\setminus A$ contains precisely the smallest elements not contained in $A$ and thus $B_i\setminus A$ must contain at least one larger element meaning that $B_i$ is lexicographically after $B$.
The condition that $A\not\subset B_i$ for all $i$ such that $z_{B_i}>z_B$ implies that $z_Az_{B_i}$ is the leading term of a polynomial in $I_n$.
As in our previous cases, this leads to the calculation
	\begin{align*}
	&S(\hat{e}_{r,k},z_Az_{[k]}-z_{A\cap [k]}z_{A\cup [k]})= \\
        &z_Az_B^r+\sum_j z_A z_{B_j}^r+\sum_m z_{A\cap B_m}z_{B_m}^{r-1}z_{A\cup B_m}
	\end{align*}
where the first sum is over all $j$ such that $|B_j|=|B|$, $B\neq B_j$, and $A\subset B_j$. The second sum is over all $m$ such that $A$ and $B_m$ are a Sperner 2-pair with $A\cap B_m\neq \emptyset$.
Notice that we know that $|B|=k$ and $B\neq [k]$ which says that there is at least some subset $\{j_1,\ldots,j_t\}\subset B$ such that $j_i>k$ for all $i$ and the defining condition $[\ell]\not\subset A$, $[\ell]\subset B$ and $B\setminus A\subset [\ell]$ implies that $j_i\in A$ for some $i$. Thus, $|A\cup [k]|>k=|B|$.
Ergo, we have $z_{A}z_{B}^r$ as the leading term.	

Our final case is to obtain the terms listed in the seventh bullet of Proposition~\ref{LT}, i.e. those of type $z_{A_1}z_{A_2}^rz_{A_3}$ where $A_1\subset A_2\subset A_3$ and $\max(A_2\setminus A_1)<\min(A_3\setminus A_2)$ with $A_2\neq [j]$ for all $j$.
Assume that we have added all prior $S$-polynomials to the generators of $J_{r,n}$.
We consider the $S$-polynomial for the elements $z_{A_2}z_{A_1\cup(A_3\setminus A_2)}-z_{A_1}z_{A_3}$ and the generator from \eqref{term} given by $z_{A_2}^{r+1}+\sum_j z_{A_2\cap C_j}z_{C_j}^{r-1}z_{A_2\cup C_j}$ where $|C_j|=|A_2|=k$, $A_2\neq C_j$, and $A_2\cap C_j\neq \emptyset$.
Let $B:=A_1\cup(A_3\setminus A_2)$ for convenience of notation, and compute:

	\begin{align*}
	& \displaystyle S\left(z_{A_2}^{r+1}+\sum_j z_{A_2\cap C_j}z_{C_j}^{r-1}z_{A_2\cup C_j},z_{A_2}z_{B}-z_{A_1}z_{A_3} \right)\\ 
= & \displaystyle \frac{z_{A_2}^{r+1}z_{B}}{z_{A_2}^{r+1}}\left(z_{A_2}^{r+1}+\sum_j z_{A_2\cap C_j}z_{C_j}^{r-1}z_{A_2\cup C_j}\right)\\
	  & \hspace{45mm} \displaystyle -\frac{z_{A_2}^{r+1}z_{B}}{z_{A_2}z_{B}}\left(z_{A_2}z_{B}-z_{A_1}z_{A_3} \right)\\
	 = & \displaystyle  z_B\sum_j z_{A_2\cap C_j}z_{C_j}^{r-1}z_{A_2\cup C_j}+z_{A_1}z_{A_2}^r z_{A_3}\\
	 = & \displaystyle z_{A_1\cup(A_3\setminus A_2)}\sum_j z_{A_2\cap C_j}z_{C_j}^{r-1}z_{A_2\cup C_j}+z_{A_1}z_{A_2}^rz_{A_3}
	\end{align*}

We now wish to show the $z_{A_1}z_{A_2}^r z_{A_3}$ is the leading term. 
Consider the terms involving $C_j$. There are three possible cases
	\begin{enumerate}[{\bf 1.}]
	\item $|A_2\cup C_j|>|A_3|$
	\item $|A_2\cup C_j|<|A_3|$
	\item $|A_2\cup C_j|=|A_3|$
	\end{enumerate}
which we consider individually.

\noindent {\bf Case 1:} If we have that $|A_2\cup C_j|>|A_3|$, then we have $z_{A_1}z_{A_2}^r z_{A_3}>  z_{A_2\cap C_j} z_{A_1\cup(A_3\setminus A_2)}z_{C_j}^{r-1} z_{A_2\cup C_j}$ immediately by the definition of graded reverse lexicographic order. 

\noindent {\bf Case 2:} Suppose that we have $|A_2\cup C_j|<|A_3|$. Note that this implies that there exists $x\in A_3$ such that $x\not\in A_2\cup C_j$ and hence $x\in A_1\cup(A_3\setminus A_2)$. We also have $y\in A_2\cup C_j$ such that $y\not\in A_1\cup (A_3\setminus A_2)$. Hence, we have that $A_1\cup (A_3\setminus A_2)$ and $A_2\cup C_j$ are a Sperner 2-pair. This implies that we can replace the monomial $z_{A_2\cap C_j} z_{C_j}^{r-1} z_{A_1\cup(A_3\setminus A_2)} z_{A_2\cup C_j}$ with the monomial 
	\begin{align*}
	& z_{A_2\cap C_j}z_{C_j}^{r-1}z_{(A_1\cup(A_3\setminus A_2))\cap (A_2\cup C_j)}z_{(A_1\cup(A_3\setminus A_2))\cup (A_2\cup C_j)}\\
	= & z_{A_2\cap C_j}z_{C_j}^{r-1}z_{A_1\cup (C_j\cap(A_3\setminus A_2))}z_{A_3\cup C_j}
	\end{align*}
	
It is clear that $|A_3\cup C_j|\geq |A_3|$. 
If the inequality is strict, then we are done. 
If $A_3\cup C_j=A_3$, note that $C_j\subset A_3$ and that $C_j\cap (A_3\setminus A_2)\neq \emptyset$ since $|C_j|=|A_2|$.
We will now consider the variable $z_{A_1\cup (C_j\cap(A_3\setminus A_2))}$. We note that two subcases arise:
	\begin{enumerate}[{\bf {2.}i.}]
	\item $A_1\cup(C_j\cap (A_3\setminus A_2))=A_1\cup C_j$ (equivalently $C_j\cap A_1=C_j\cap A_2$)
	\item $A_1\cup(C_j\cap (A_3\setminus A_2))$ and $A_2\cap C_j$ are a Sperner 2-pair. 
	\end{enumerate}

\underline{Subcase 2.i}: Note that $|A_1\cup C_j|\geq |A_2|$ with equality occurring if $A_1\cup C_j=C_j$. If the inequality is strict, we are done. If $A_1\cup C_j=C_j$, then $|C_j|=|A_2|$, but since $C_j\cap A_1=C_j\cap A_2$ and $C_j\cap(A_3\setminus A_2)\neq \emptyset$, the condition $\max(A_2\setminus A_1)<\min(A_3\setminus A_2)$ implies that $A_2$ is lexicographically before $C_j$, which is desired. 

\underline{Subcase 2.ii}: The existence of such a Sperner 2-pair allows us to replace the monomial through division by 
	\begin{align*}
	&  z_{(A_1\cup(C_j\cap (A_3\setminus A_2)))\cap (A_2\cap C_j)}z_{(A_1\cup(C_j\cap (A_3\setminus A_2)))\cup (A_2\cap C_j)}z_{C_j}^{r-1}z_{A_3}\\ 
	= & z_{(A_1\cup(C_j\cap (A_3\setminus A_2)))\cap (A_2\cap C_j)}z_{A_1\cup C_j}z_{C_j}^{r-1}z_{A_3}
	\end{align*}		

Showing the desired outcome is now identical to the argument in Subcase 2.i.

\noindent {\bf Case 3:} Suppose that we have $|A_2\cup C_j|=|A_3|$. 
In this case, it is sufficient to consider the following three plausible sub-cases.
	\begin{enumerate}[{\bf {3.}i.}]
	\item $A_2\cup C_j$ and $A_1\cup(A_3\setminus A_2)$ are a Sperner 2-pair.
	\item The subcase  3.i. is false, but $A_2\cap C_j$ and $A_1\cup(A_3\setminus A_2)$ are a Sperner 2-pair.
	\item $A_2\cap C_j$, $A_2\cup C_j$, and $A_1\cup(A_3\setminus A_2)$ have no Sperner 2-pairs between them. 
	\end{enumerate}

\underline{Subcase 3.i}: Suppose we have that the sets  $A_2\cup C_j$ and $A_1\cup(A_3\setminus A_2)$ are a Sperner 2-pair. This means that via division, we can replace the existing monomial $z_{A_2\cap C_j} z_{C_j}^{r-1} z_{A_1\cup(A_3\setminus A_2)} z_{A_2\cup C_j}$ with the monomial
	\begin{align*}
	 & z_{A_2\cap C_j}z_{C_j}^{r-1}  z_{(A_2\cup C_j)\cap (A_1\cup(A_3\setminus A_2))}z_{(A_2\cup C_j)\cup (A_1\cup(A_3\setminus A_2))}\\ 
	 = & z_{A_2\cap C_j}z_{C_j}^{r-1}  z_{(A_2\cup C_j)\cap (A_1\cup(A_3\setminus A_2))}z_{A_3\cup C_j}
	\end{align*}

By virtue of the Sperner 2-pair assumptions, we have that there exists $x\in C_j$ such that $x\not\in A_3$, which yields $|A_3\cup C_j|>|A_3|$ and hence
	\[
	z_{A_1}z_{A_2}^rz_{A_3}>z_{A_2\cap C_j}z_{C_j}^{r-1} z_{(A_2\cup C_j)\cap (A_1\cup(A_3\setminus A_2))}z_{A_3\cup C_j}
	\]
and we are done.
	
\underline{Subcase 3.ii}: Suppose that $A_2\cap C_j$ and  $A_1\cup (A_3\setminus A_2)$ are a Sperner 2-pair, but that $A_2\cup C_j$ and $A_1\cup(A_3\setminus A_2)$ are not. Then note that we have  $A_1\cup(A_3\setminus A_2)\subset A_2\cup C_j$, which implies that $A_3\setminus A_2\subset C_j$, and hence $A_2\cup C_j=A_3$ by the cardinality assumption. Now, by the existence of the Sperner 2-pair, we can replace via division the existing monomial $z_{A_2\cap C_j} z_{C_j}^{r-1} z_{A_1\cup(A_3\setminus A_2)} z_{A_2\cup C_j}$ with the monomial
	\[
	z_{(A_2\cap C_j)\cap(A_1\cup(A_3\setminus A_2))}z_{(A_2\cap C_j)\cup(A_1\cup(A_3\setminus A_2))}z_{C_j}^{r-1} z_{A_3}
	\]
Moreover, notice that $C_j\subseteq ((A_2\cap C_j)\cup(A_1\cup(A_3\setminus A_2)))$. If the equality is strict, we have that $|A_2|<| ((A_2\cap C_j)\cup(A_1\cup(A_3\setminus A_2)))|$ and we are done. If we have equality, then we know $|A_2|=| ((A_2\cap C_j)\cup(A_1\cup(A_3\setminus A_2)))|$. By the assumption that $\max(A_2\setminus A_1)<\min(A_3\setminus A_2)$, this implies that $A_2$ is lexicographically before $((A_2\cap C_j)\cup(A_1\cup(A_3\setminus A_2)))$. Thus we will have
	\[
	z_{A_1}z_{A_2}^r z_{A_3}> z_{(A_2\cap C_j)\cap(A_1\cup(A_3\setminus A_2))}z_{(A_2\cap C_j)\cup(A_1\cup(A_3\setminus A_2))}z_{C_j}^{r-1}z_{A_3}
	\] 
which is as desired. 

\underline{Subcase 3.iii}: Suppose that the sets $A_2\cap C_j$, $A_2\cup C_j$, and $A_1\cup(A_3\setminus A_2)$ have no Sperner 2-pairs between them. This implies the following containment 
	\[
	A_2\cap C_j\subset A_1\cup(A_3\setminus A_2) \subset A_2\cup C_j=A_3
	\]
because $A_2\cap C_j\subseteq A_1$ and $A_3\subseteq A_2\cup C_j$, which follows from the necessary containment and the fact that these sets have the same cardinality. 
These observations allow us to conclude that $C_j\subseteq A_1\cup(A_3\setminus A_2)$. If the containment is strict, we have that $|A_1\cup(A_3\setminus A_2)|>|A_2|$ and we are done. If equality holds, we have $|A_1\cup(A_3\setminus A_2)|=|A_2|$. However, the assumed condition that $\max(A_2\setminus A_1)<\min(A_3\setminus A_2)$ implies that $A_2$ is lexicographically before $A_1\cup(A_3\setminus A_2)$. Thus, we have that 
	\[
	z_{A_1}z_{A_2}^rz_{A_3}>z_{A_2\cap C_j}z_{A_1\cup(A_3\setminus A_2)}z_{C_j}^{r-1}z_{A_3}
	\]		 		
which is our desired result. 

Given all of the above, we can conclude that 
	\[
	S\left(z_{A_2}^r+\sum_j z_{A_2\cap C_j}z_{C_j}^{r-1}z_{A_2\cup C_j},z_{A_2}z_{B}-z_{A_1}z_{A_3} \right)=z_{A_1}z_{A_2}^rz_{A_3}+p_{A_1A_2^rA_3}
	\]		
where $p_{A_1A_2^rA_3}$ is a polynomial with $\LT(p_{A_1A_2^rA_3})<z_{A_1}z_{A_2}^rz_{A_3}$.

We have now shown that all of our desired leading terms appear through the optimized Buchberger Algorithm. 
Because of our previous dimension calculation for $T_n/J_{r,n}$, we know that no additional leading terms can result form further computations, thus we have a Gr\"obner basis. 
\end{proof}

We have thus established Theorem \ref{basistheorem}, as it follows immediately from Theorem~\ref{Mac}, Proposition \ref{LT}, and Theorem \ref{gb}.

%%%%%%%%%%%%%%%%%%%%%%%%%%%%%%%%%%%%%%%%%%%%%%%%%%%%%

\section{Combinatorial identities}\label{proofs}

We will now compute multigraded Hilbert series to prove Theorems~\ref{EMI} and~\ref{neg}. Recall from Section \ref{Intro} that we can define a Hilbert series with respect to a $\Z^m$-grading for any $m\geq 1$ as in~\eqref{HilbertSeries}. We now define the $\Z^2$-grading which arises from the the defined degree on variables $\deg(z_A)=tq^{|A|}$, where we note that $\deg(z_\emptyset)=t$. We denote this bivariate Hilbert series as $\Hilb{A}{t,q}$ for a graded module $A$ of $T_n$.
It is straightforward \cite{B-B} to show that $\Hilb{R_n}{t,q}=\sum_{k\geq 0}[k+1]_q^n t^k$, which we assume for both of the following proofs.
We will use the notation $\mathcal{C}_{r,n}$ introduced in the proof of Theorem~\ref{gb}.

\begin{proof}[Proof of Theorem~\ref{EMI}]
Given that $R_n$ is Cohen-Macaulay and the elements of $invar(1,n)$ are an algebraically independent homogeneous system of parameters as argued in the proof of Theorem~\ref{gb}, we can express the Hilbert series in the form 
	\[
	\Hilb{R_n}{t,q}=\frac{\Hilb{T_n/J_{1,n}}{t,q}}{\prod_{j=0}^n(1-tq^j)}\, .
	\]
This follows because it is an elementary exercise to compute that  
	\[
	\Hilb{\mathcal{C}_{1,n}}{t,q}=\frac{1}{(1-t)(1-tq)\cdots(1-tq^n)} \, .	
	\]
To compute the numerator, we have  
	\[
	\Hilb{T_n/J_{1,n}}{t,q}=\sum_{\pi\in S_n}\deg(\api)=\sum_{\pi\in S_n}t^{\des(\pi)}q^{\maj(\pi)}
	\]
by using the basis for $T_n/J_{1,n}$ given by Theorem~\ref{basistheorem}. This completes the proof.	
\end{proof}

\begin{proof}[Proof of Theorem~\ref{neg}]
Given that $R_n$ is Cohen-Macaulay and $invar(r,n)$ is an algebraically independent homogenous system of parameters, we can express the Hilbert series as
	\[
	\Hilb{R_n}{t,q}=\frac{\Hilb{T_n/J_{r,n}}{t,q}}{(1-t)\prod_{j=1}^n(1-t^rq^{rj})}.
	\]
This follows because, as in the previous proof, it is straightforward to show that	
	\[
	\Hilb{\mathcal{C}_{r,n}}{t,q}=\frac{1}{(1-t)(1-t^r q^r)(1-t^r q^{2r})\cdots(1-t^r q^{rn})} \, .
	\]
Hence, we compute the numerator by employing the basis given in Theorem \ref{basistheorem}
	\[
	\begin{array}{rcl}
	\Hilb{T_n/J_{r,n}}{t,q} & = & \displaystyle \sum_{(\pi,X)\in \Z_r \wr S_n}\deg(b^r_{(\pi,X)})\\
	& = & \displaystyle\sum_{(\pi,X)\in \Z_r \wr S_n} t^{\des(\pi)}q^{\maj(\pi)}t^{|X|}q^{\sum_{i\in X}i}\\
	& = & \displaystyle\sum_{(\rho,\epsilon)\in \Z_r \wr S_n}t^{\ndes(\rho,\epsilon)}q^{\nmajor(\rho,\epsilon)} \, ,
	\end{array}
	\] 
completing the proof.	
\end{proof}

%%%%%%%%%%%%%%%%%%%%%%%%%%%%%%%%%%%%%%%%%%%%%%%%%%%%%%%%%%%%%%

\section{Concluding Remarks}

It is worth mentioning that when $r=1$ there is a graded $S_n$-module isomorphism between $T_n/J_{1,n}$ and $\C[x_1,x_2,\ldots, x_n]/\mathscr{I}_n$.

\begin{theorem}\label{isomorphism}
The map $\phi:T_n/J_{1,n}\to \C[x_1,x_2,\ldots, x_n]/\mathscr{I}_n$ defined by algebraically extending $z_A+J_{1,n}\mapsto\prod_{i\in A}x_i +\mathscr{I}_n$ is an $S_n$-isomorphism.
\end{theorem}

\begin{proof}
Consider $T_n/J_{1,n}$ under the $q$-grading used in the multigrading for Section~\ref{proofs}, i.e. $\deg(z_A)=|A|$. 
Let $\C[x_1,x_2,\cdots, x_n]/\mathscr{I}_n$ be graded by total degree. 
It is clear that $\phi$ respects grading, by definition. 
Moreover, it is clear that $\phi$ is an algebra isomorphism, since 
	\[
	\begin{array}{rcl}
	\phi(z_A+J_{1,n})\cdot \phi(z_B+J_{1,n}) & = & (\prod_{i\in A}x_i +\mathscr{I}_n)\cdot(\prod_{j\in B}x_j +\mathscr{I}_n)\\
	& = & (\prod_{i\in A}x_i)\cdot(\prod_{j\in B}x_j) +\mathscr{I}_n\\
	& = & \phi(z_Az_B +J_{1,n})
	\end{array}
	\]
which implies $	\phi(\api+J_{1,n})=a_\pi+\mathscr{I}_n$ for all $\pi\in S_n$.
	
Now we show that the action is preserved. 
Consider $z_A+J_{1,n}$ and $\sigma\in S_n$, and observe that 
	\[
	\begin{array}{rcl}
	\sigma\circ \phi (z_A+J_{1,n}) & = & \sigma\left(\prod_{i\in A}x_i \right)+\mathscr{I}_n\\
	& = & \prod_{i\in A}x_{\sigma(i)}+\mathscr{I}_n\\
	& = &  \prod_{i\in \sigma(A)}x_i+\mathscr{I}_n\\
	& = & \phi (z_{\sigma(A)}+J_{1,n})\\
	& = & \phi \circ \sigma (z_{A}+J_{1,n}) \, .
	\end{array}	
	\]		
\end{proof}
It would be interesting to determine if the representation-theoretic results of \cite{A-B-R} are easier to establish in the context of $T_n/J_{1,n}$ rather than $\C[x_1,x_2,\ldots, x_n]/\mathscr{I}_n$.

%%%%%%%%%%%%%%%%%%%%%%%%%%%%%%%%%%%%%%%%%%%%%%%%%%%%%%%%%%%%%%
%%%%%%%%%%%%%%%%%%%%%%%%%%%%%%%%%%%%%%%%%%%%%%%%%%%%%%%%%%%%%%%%%%%%%

%Bibliography (Because McCabe can't use BibTex) hahaha...

\end{document}